\documentclass[final]{siamltex}
\usepackage{amsmath, amssymb, graphicx, epstopdf, algpseudocode, algorithm}
\newtheorem{remark}{Remark}

\newcommand{ \rn }[1]{\mathbb{R}^{#1}}

\title{Numerical Recovery of Source Singularities via the Radiative Transfer Equation with Partial Data \thanks{This work was supported by NSF grant DMS-0838212.}}

\author{Mark Hubenthal\thanks{Department of Mathematics and Statistics, University of Jyv\"askyl\"a ({\tt john.m.hubenthal@jyu.fi}) and Department of Mathematics, University of Washington}}

\begin{document}
\maketitle
\begin{abstract}The inverse source problem for the radiative transfer equation is considered, with partial data. Here we demonstrate numerical computation of the normal operator $X_{V}^{*}X_{V}$ where $X_{V}$ is the partial data solution operator to the radiative transfer equation. The numerical scheme is based in part on a forward solver designed by F. Monard and G. Bal. We will see that one can detect quite well the visible singularities of an internal optical source $f$ for generic anisotropic $k$ and $\sigma$, with or without noise added to the accessible data $X_{V}f$. In particular, we use a truncated Neumann series to estimate $X_{V}$ and $X_{V}^{*}$, which provides a good approximation of $X_{V}^{*}X_{V}$ with an error of higher Sobolev regularity. This paper provides a visual demonstration of the authors' previous work in recovering the microlocally visible singularities of an unknown source from partial data. We also give the theoretical analysis necessary to justify the smoothness of the remainder when approximating the normal operator.\end{abstract}



\pagestyle{myheadings}
\thispagestyle{plain}
\markboth{M. HUBENTHAL}{NUMERICAL RECOVERY OF SOURCE SINGULARITIES}

\section{Introduction}
The Radiative Transfer Equation (also sometimes referred to as the linear transport or linear Boltzmann equation) is often used to model the propagation of particles that exhibit absorption and scattering in various contexts, including the behavior of photons within biological tissues or of neutrons in a reactor. In this paper, we will consider an inverse problem relevant to optical molecular imaging (OMI), which is a physically safe imaging modality that has seen some recent advances in research. In this application, biochemical markers can be used to detect the presence of specific molecules or genes, and suitably designed markers could potentially identify diseases before phenotypical symptoms even appear. Such markers are typically light-emitting molecules, such as fluorophores or luminophores, which bind to the desired molecule to be detected. In contrast to Single Positron Emission Computed Tomography (SPECT),  Positron Emission Tomography (PET), or Magnetic Resonance Imaging (MRI), optical markers emit low-energy near-infrared photons that are relatively harmless to human tissue. Because of their low-energy level, the photons scatter before they are measured. Further specifics can be found in the bioengineering literature such as \cite{be1,be2,be3, weissleder1,weissleder2,ntziachristos}. 

The inverse problem we consider consists of reconstructing the spatial distribution of a radiative light source from measurements of outgoing photon intensities at the boundary of the medium. In many applications, the propagation of photons emitted can be modeled as inverse source problems of time-independent radiative transfer equations. We will assume that the optical parameters of the underlying medium are known, so that the problem of determining the source is feasible. In practice, such optical parameters do not vary too much in biological tissue, and not knowning their true values does not significantly effect one's ability to detect edges in the image. It is shown in \cite{inversesource} that under mild assumptions on the scattering and absorption parameters of the medium this is possible. Other work on the inverse source problem for the RTE under varying assumptions can be found in \cite{guillaume, guillaume2, sharafutdinov, panchenko, larsen}, and further background on optical tomography can be found in \cite{uhlmanot, arridge}. In the case of partial data, \cite{hubenthal} shows that one can recover the \textit{visible} singularities of the source $f$, which is a particular subset of $\mathrm{WF}(f)$. As such, the primary focus of this paper is to provide a numerical scheme by which to detect such singularities. We now describe more precisely the mathematical problem.

We assume $\Omega$ to be a bounded domain in $\rn{n}$ with smooth boundary and outer unit normal vector $\nu(x)$. As in \cite{inversesource,hubenthal}, we also assume that the data is given on the boundary of a larger strictly convex domain $\Omega_{1} \Supset \Omega$. The radiative transfer equation on the domain $\Omega$ with internal source $f$ and zero incoming illuminations is given by
\begin{align}
\theta \cdot \nabla_{x}u(x,\theta) + \sigma(x,\theta)u(x,\theta) - \int_{\mathbb{S}^{n-1}}k(x,\theta,\theta')u(x,\theta')\,d\theta' & = f(x), \notag\\
\quad u|_{\partial_{-}S\Omega} & = 0, \label{transport}
\end{align}
where the absorption $\sigma$ and the collision kernel $k$ are functions with regularity specified later, the solution $u(x,\theta)$ gives the intensity of photons at $x$ moving in the direction $\theta$, and $\partial_{\pm}S\Omega$ is the set of points $(x,\theta) \in \partial \Omega \times \mathbb{S}^{n-1}$ such that $\pm \nu(x) \cdot \theta > 0$. That is, $\partial_{\pm}S\Omega$ is the set of points $(x,\theta) \in \partial \Omega \times \mathbb{S}^{n-1}$ such that $\theta$ is pointing outward or inward, respectively. The source term $f$ will be assumed to depend on $x$ only for our purposes. We also recall, given the physical model, that equation (\ref{transport}) is only applicable at a single frequency, as the parameters $\sigma$ and $k$ typically vary widely with different frequencies. It is also important that the photons' energy be relatively low, since for high energy photons each scattering event is often accompanied by an energy change. This is discussed briefly in \cite{hubenthal}, and we mention it primarily to point out the limitations imposed by the assumptions of the model.

In the case of full data, we have boundary measurements
\begin{equation}
Xf(x,\theta) = u\vert_{\partial_{+}S\Omega}.
\end{equation}
In \cite{inversesource}, it is shown that for an open, dense set of absorption and scattering coefficients $(\sigma, k) \in C^{2}(\overline{\Omega} \times \mathbb{S}^{n-1}) \times C^{2}(\overline{\Omega} \times \mathbb{S}^{n-1} \times \mathbb{S}^{n-1})$, one can recover $f \in L^{2}(\Omega)$ uniquely from boundary measurements $Xf$ on all of $\partial_{+}S\Omega$. To set up the case of partial data, first let $V \subset \partial_{+}S\Omega$ be open and let $\widetilde{V} \Subset V$. Let $\chi_{V}\in C_{0}^{\infty}(\partial_{+}S\Omega)$ be a smooth cutoff function such that $\chi_{V}(x,\theta) \equiv 1$ for $(x,\theta) \in \widetilde{V}$ and $\chi_{V}(x,\theta) \equiv 0$ for $(x,\theta) \notin V$. The boundary measurements are then given by 
\begin{equation}
X_{V}f(x,\theta) = \chi_{V}(x,\theta) u\vert_{\partial_{+}S\Omega}.
\end{equation}
To make notation a bit simpler, if $V = \partial_{+}S\Omega$ (complete data) we will just write $X$, since in this case $X_{V} = X$.

We now consider the complementary problem of numerical computation of solutions to (\ref{transport}) and using the accessible part of the synthethized data to visualize the results of \cite{hubenthal}. To this end, we employ a technique from \cite{monard} that uses rotations applied to the parameters in the spectral domain to help eliminate the ray effect, which is a byproduct of the discrete ordinates method. It should be noted that in principle such a method could be applied in any dimension, but for the actual computations we will restrict ourselves to the two dimensional case. We also refer the reader to \cite{hongkai} for another approach to solving the direct transport equation using finite element methods, which we do not use here.

Our main goal is to approximate the operators $X^{*}Xf$ and $X_{V}^{*}X_{V}f$ in the two-dimensional case. Ignoring the technical details for the moment, the idea is to utilize a Neumann series to approximate the forward solution operator $X_{V}$ and its adjoint $X_{V}^{*}$. As such, one must truncate the series to a finite number of terms when computing $X_{V}^{*}X_{V}$, and the difference of the approximation from the true function $X_{V}^{*}X_{V}f$ will have some (higher) Sobolev regularity. The only theoretical result herein is Lemma \ref{approxnormalsmooth}, which essentially asserts that, if $m_{1}$ terms are computed in the approximation of $X$ and $m_{2}$ terms are computed in the approximation of $X^{*}$, then $[X^{*}X]_{\mathrm{approx}}f - X^{*}Xf \in H^{l + m+1}(\Omega)$, where $m  = \min\{ m_{1},m_{2}\}$ and $f \in H^{l}(\Omega)$.

In \S \ref{statementofresults}, we review the background related to recovering the visible singularities of a source $f$ given partial data of its outgoing intensity on the boundary. This is summarized by Theorem 2.2 and Corollary 6.4 of \cite{hubenthal}. In \S \ref{numericalmethod} we describe the algorithm used to compute the forward solution $u$ to (\ref{transport}). Following, in \S \ref{ApproxNormal} we detail the similar approach used to approximate $X_{V}^{*}$. Finally, we present some numerical computations for a few different types of sources with and without noise added to $X_{V}f$ in \S \ref{examples}.  Section \ref{smoothnessanalysis} proves the main technical estimate of this work, and Appendix \ref{ap:compsing} provides some necessary results for iterated weakly singular-type integral operators.

\section*{Acknowledgements}The author wishes to thank Fran\c cois Monard for some enlightening conversations related to the computational aspects of this work, and to thank the referees for their valuable comments and suggestions. This research was primarily done while at the University of Washington. However, much of the work in Appendix \ref{ap:compsing} was completed during a visit to Finland under the advisement of Mikko Salo, with whom it was very informative to talk about this and related problems.

\section{Background and Statement of Main Lemma \label{statementofresults}}
Let us establish some standard notation. We set
\begin{equation}
T_{0} = \theta \cdot \nabla_{x}, \quad T_{1} = T_{0} + \sigma, \quad T = T_{0} + \sigma - K,
\end{equation}
where $\sigma$ denotes the operation of multiplication by $\sigma(x,\theta)$, and $K$ is defined by
\begin{equation}
Kf(x,\theta) = \int_{\mathbb{S}^{n-1}}k(x,\theta,\theta')f(x,\theta')\,d\theta'.
\end{equation}
We also recall from \cite{inversesource} that a particular left inverse for $T_{1}$ is given by
\begin{equation}
[T_{1}^{-1}f](x,\theta) := \int_{-\infty}^{0}\exp\left(  -\int_{s}^{0}\sigma(x + \tau \theta, \theta)\,d\tau \right)f(x + s\theta,\theta)\,ds, \label{t1inverse}
\end{equation}
which can be verified by noting that $E(x,\theta) = \exp\left( -\int_{\rn{+}} \sigma(x+s\theta,\theta)\,ds \right)$ is an integrating factor for $T_{1}$.

In order to compute the normal operator $X_{V}^{*}X_{V}f$, we first need to generate the solution to the radiative transport equation (\ref{transport}). For the purposes of the numerical implementation, we will assume that the Neumann series expansion of $(\mathrm{Id} - KT_{1}^{-1})^{-1}$ converges, which is the case when $\|KT_{1}^{-1}\|_{L^{2} \to L^{2}} < 1$. In this case we have
\begin{align*}
u &= T_{1}^{-1}(\mathrm{Id} - KT_{1}^{-1})^{-1}Jf\\
& = T_{1}^{-1}\sum_{m=0}^{\infty}(KT_{1}^{-1})^{m}Jf.
\end{align*}
If we set $u_{0}(x,\theta) := [T_{1}^{-1}f](x,\theta)$ and set $u_{l} := T_{1}^{-1}Ku_{l-1}$ for $l \geq 1$, then
$u = \sum_{l\geq 0}u_{l}$. In practice, it is quite simple to compute the scattering term $Ku_{l-1}$ at each iteration. In order to compute $T_{1}^{-1}$ we will use a simple first order Euler method to solve the associated differential equation to which $T_{1}^{-1}$ is the solution operator:
\begin{align*}
\theta \cdot \nabla_{x}u(x,\theta) + \sigma(x,\theta)u(x,\theta) & = g(x,\theta), \quad (x,\theta) \in \Omega \times \mathbb{S}^{n-1}\\
u \vert_{\partial_{-}S\Omega} & \equiv 0.
\end{align*}

For a given source $f \in L^{2}(\Omega)$, recall the \textit{microlocally visible set} corresponding to partial measurements on $\partial_{+}S\Omega_{1}$, given by
 \begin{equation}
\mathcal{M}' := \{ (x,\xi) \in T^{*}\Omega \, | \, \exists \theta \in \mathbb{S}^{n-1} \textrm{ such that } \theta \cdot \xi = 0 \textrm{ and } \chi_{V}^{\#}(x,\theta) \neq 0 \}.\label{microvisibleset}
\end{equation}
Here $\chi_{V}^{\#}(x,\theta)$ is the extension of $\chi_{V}:\partial_{+}S\Omega_{1} \to \mathbb{R}$ to $\Omega_{1} \times \mathbb{S}^{n-1}$ defined by $\chi_{V}^{\#}(x,\theta) = \chi_{V}(x + \tau_{+}(x,\theta)\theta, \theta)$.

It will be useful to define the so-called Sobolev wavefront sets of a given order $s$ (see \cite{quintoxraysurvey}).
\begin{definition}
Let $f \in \mathcal{D}'(\rn{n})$. Then $f$ is in $H^{s}$ \textit{microlocally near} $(x_{0},\xi_{0})$ if and only if there is a cut-off function $\psi \in \mathcal{D}(\rn{n})$ with $\psi(x_{0}) \neq 0$ and a function $u(\xi)$ homogeneous of degree zero and smooth on $\rn{n} \setminus 0$ and with $u(\xi_{0}) \neq 0$ such that the product $u(\xi) \widehat{(\psi f)}(\xi) \in L^{2}(\rn{n}, (1+|\xi|^{2})^{s})$. The $H^{s}$ \textit{wavefront set of } $f$, $\mathrm{WF}^{s}(f)$, is the complement of the set of $(x_{0},\xi_{0})$ near which $f$ is microlocally in $H^{s}$.\end{definition}

With this notation, the standard $C^{\infty}$ wavefront set can be written as $\mathrm{WF}^{\infty}(f)$. We now restate the main result of \cite{hubenthal} for recovering singularities of a source from partial data, albeit in a slightly more compact way using the $H^{s}$ wavefront sets defined above.
\begin{theorem}[\cite{hubenthal}, Theorem 2.2] Suppose $f \in L^{2}(\Omega)$ and let $l$ be a positive integer. There exists an open dense set $\mathcal{O}_{l}$ of pairs $(\sigma, k) \in  C^{\infty}(\overline{\Omega} \times \mathbb{S}^{n-1}) \times  C^{\infty}(\overline{\Omega}_{x} \times \mathbb{S}_{\theta'}^{n-1} \times \mathbb{S}_{\theta}^{n-1})$ depending on $l$ such that given $(\sigma, k) \in \mathcal{O}_{l}$, (\ref{transport}) has a unique solution $u \in L^{2}$, and
\begin{equation}
(z,\xi) \notin \mathrm{WF}^{\infty}(X_{V}^{*}X_{V}f) \Longrightarrow (z,\xi) \notin  \mathrm{WF}^{l}(f), \quad \forall (z,\xi) \in \mathcal{M}'.
\end{equation}
\label{microtheorem}
\end{theorem}
\begin{remark}In \cite{hubenthal} Theorem \ref{microtheorem} is stated with the conclusion that $(z,\xi) \notin \mathrm{WF}^{\infty}(f + v)$, where the remainder $v$ is an $H^{l}$ function with an explicit formula in terms of the measurement operator.\end{remark}

When $\|KT_{1}^{-1}\|$ has small enough norm as an operator from $L^{2}(\Omega \times \mathbb{S}^{n-1}) \to L^{2}(\Omega \times \mathbb{S}^{n-1})$, it turns out that the dependence on $l$ in the Theorem \ref{microtheorem} goes away. 
\begin{corollary}[\cite{hubenthal}, Corollary 6.4] Suppose that $\|KT_{1}^{-1}\| < 1$. Then there is a dense open set $\mathcal{O}$ of pairs $(\sigma, k) \in C^{\infty} \times C^{\infty}$ for which (\ref{transport}) has a unique solution $u \in L^{2}$ for all $f \in L^{2}(\Omega)$, and
\begin{equation*}
(z,\xi) \notin \mathrm{WF}^{\infty}(X_{V}^{*}X_{V}f) \Longrightarrow (z,\xi) \notin \mathrm{WF}^{\infty}(f), \quad \forall (z,\xi) \in \mathcal{M}'.
\end{equation*}
\end{corollary}

With this in mind, we can state the main theoretical estimate of this work, which justifies our approach in the numerical computations. The proof is given in \S \ref{smoothnessanalysis}.
\begin{lemma}
Suppose $f \in H^{l}(\Omega)$ and $\sigma, k \in C^{\infty}$. Let $m_{1},m_{2} \geq 0$. Then
\begin{equation*}
X_{V}^{*}X_{V}f - \sum_{j=0}^{m_{2}}\sum_{i=0}^{m_{1}}[\chi_{V}R_{+}T_{1}^{-1}(KT_{1}^{-1})^{j}J]^{*}\chi_{V}R_{+}T_{1}^{-1}(KT_{1}^{-1})^{i}Jf \in H^{l+m+1}(\Omega_{1})
\end{equation*}
where $m = \min\{m_{1},m_{2}\}$. \label{approxnormalsmooth}
\end{lemma}
\begin{remark}Note that the partial Neumann series expansion in Lemma \ref{approxnormalsmooth} might not converge to $X_{V}^{*}X_{V}f$ as $\min\{m_{1},m_{2}\}\to \infty$, even if $(\sigma, k)$ are such that (\ref{transport}) is well-posed. In such cases, the error above still has the stated smoothness but possibly becomes unbounded as $m_{1},m_{2} \to \infty$.\end{remark}

It is important to realize that although $X_{V}^{*}X_{V}f$ contains information of the singularities of $f$, its smoothing properties change the order of such singularities. Recall from \cite{hubenthal} that $X_{V}^{*}X_{V}f = I_{\sigma, V}^{*}I_{\sigma,V}f + \mathcal{L}_{V}f$ where
\begin{align*}
I_{\sigma, V}f(x,\theta) & := \chi_{V}(x,\theta) I_{\sigma}f(x,\theta) \\
& = \chi_{V}(x,\theta)  \int_{\tau_{-}(x,\theta)}^{0}E(x+t\theta,\theta)f(x+t\theta)\,dt, \quad (x,\theta) \in \partial_{+}S\Omega_{1}
\end{align*}
is the partial data attenuated X-ray transform, and $\mathcal{L}_{V}$ is a $\Psi$-DO of order $-2$ plus an integral operator $R:L^{2}(\Omega) \to H^{l}(\Omega_{1})$ with $l \geq 1$ arbitrary ($R = 0$ if $\|KT_{1}^{-1}\| < 1$). Moreover, from \cite{xraygeneric} $I_{\sigma, V}^{*}I_{\sigma,V}$ is a $\Psi$-DO of order $-1$, elliptic on $\mathcal{M}'$, and its principal symbol is explicitly given by
\begin{equation}
\rho(x,\xi) = 2\pi \int_{\theta \cdot \xi = 0}|E(x,\theta)|^{2}|\chi_{V}^{\#}(x,\theta)|^{2}\,dS(\theta). \label{eq:Iprincipalsymbol}
\end{equation}
Thus, in order to obtain the singularities of $f$ in $\mathcal{M}'$ with the correct order, one could attempt to compute explicitly a parametrix $Q$ of $I_{\sigma, V}^{*}I_{\sigma, V}$ by defining
\begin{equation*}
Qg = \mathcal{F}^{-1}(\varphi(x,\xi)\rho(x,\xi)^{-1}\widehat{g}),
\end{equation*}
where $\varphi(x,\xi) \equiv 1$ on some compact subset of $\mathcal{M}'$ and vanishes outside of $\mathcal{M}'$ ($\mathcal{F}$ denotes the Fourier transform). Applying the $Q$ from above to $X_{V}^{*}X_{V}f$ will yield $f$ plus some error term $Sf$ which is order $1$ more smooth than $f$. In this paper we do not go farther than computing $X_{V}^{*}X_{V}$, and thus the results herein are primarily a qualitative realization of the theory from \cite{hubenthal}. However, it may be of interest to consider the computation of such parametrices in the future.

\section{Numerical Method for the Direct Problem \label{numericalmethod}}
We now present a detailed summary of the method used to solve (\ref{transport}) for $n=2$. Note that at this stage it is not important that the source $f$ be isotropic, but such an assumption will be used when computing the normal operator later on. As in \cite{monard} we use the source iteration method, which requires one to solve problems of the form
\begin{align}
& \theta \cdot \nabla_{x}u + \sigma(x,\theta)u = g(x,\theta) \notag\\
& u \vert_{\partial_{-}S\Omega} = 0. \label{transportiteration}
\end{align}

Without loss of generality, we may take $\Omega = \mathbb{D} \subset \rn{2}$ where $\mathbb{D}$ is the unit disk, and assume that $\sigma,k$ and $f$ are all supported compactly in $\mathbb{D} \times \mathbb{S}^{1}$ or $\mathbb{D}$ as appropriate. Such an assumption can be justified by finding a ball $B(0,R)$ large enough and rescaling the coordinates accordingly. The main advantage here from a numerical standpoint is that $\Omega$ remains invariant under rotations. Now, for actually computing $T_{1}^{-1}$, it will be easier to have boundary conditions on a cartesian domain. To that end, for each $\eta \in (0,2\pi)$ denote $\theta = \theta(\eta) = (\cos(\eta),\sin(\eta))$ and let $\theta^{\perp} = \theta(\eta)^{\perp}$ be the counterclockwise rotation of $\theta$ by $\pi/2$. That is $\theta(\eta)^{\perp} = (-\sin{\eta}, \cos{\eta})$. For each $\eta \in (0,2\pi)$ define the $\eta$-dependent square
\begin{equation*}
C_{\eta} = \{x \in \rn{2} \textrm{ such that }|x \cdot \theta| < 1 \textrm{ and }|x \cdot \theta^{\perp}| < 1\}.
\end{equation*}
In short, $C_{\eta}$ is the square of side length $2$ rotated by an angle $\eta$. The corresponding incoming and outgoing sets (analogous to $\partial_{\pm}S\Omega$ are given by
\begin{equation*}
\Gamma_{\pm, \eta} = \{x \in \partial C_{\eta} \textrm{ such that } x \cdot \theta = \pm 1 \textrm{ and } |x \cdot \theta^{\perp}|< 1 \}.
\end{equation*}
At the heart of the rotational method, we perform a rotational change coordinates so that the derivative in the direction $\theta$ becomes $\partial_{x}$. Fix an angle $\eta$ and for $x = (\mathbf{x}, \mathbf{y}) \in [-1,1]^{2}$ define
\begin{align}
u_{\eta}(\mathbf{x}, \mathbf{y}) & = [u]_{\eta}(\mathbf{x}, \mathbf{y}) := u(\mathbf{x}\cos{\eta} - \mathbf{y}\sin{\eta}, \mathbf{x}\sin{\eta} + \mathbf{y}\cos{\eta}, \eta)\\
v_{\eta}(\mathbf{x}, \mathbf{y}) & = [v]_{\eta}(\mathbf{x}, \mathbf{y}) := v(\mathbf{x}\cos{\eta} - \mathbf{y}\sin{\eta}, \mathbf{x}\sin{\eta} + \mathbf{y}\cos{\eta})
\end{align}
where $u,v$ are functions on $\rn{2}$ and $\rn{2} \times \mathbb{S}^{1}$, respectively. The bracket notation will only be used for functions which already have a subscript. With this change of variables, we can rewrite (\ref{transportiteration}) as
\begin{align}
& \frac{\partial}{\partial x_{1}} u_{\eta}(\mathbf{x}, \mathbf{y}) + \sigma_{\eta}(\mathbf{x}, \mathbf{y})u_{\eta}(\mathbf{x}, \mathbf{y}) = g_{\eta}(\mathbf{x}, \mathbf{y}),\label{rotatedtransport}\\
& u_{\eta}(-1,\mathbf{y}) = 0 \qquad \textrm{ on }\Gamma_{-,\eta}\notag
\end{align}
We remark that for a general (possibly non-zero) function $h \in L^{1}(\partial_{-}S\mathbb{D})$, we would have the boundary condition $\widetilde{h}_{\eta}$ on $\Gamma_{-,\eta}$, obtained by projecting $h$ onto $\Gamma_{-,\eta}$ via the relation $\widetilde{h}_{\eta}(x_{1},x_{2}) = h( P_{-}^{-1}(x_{1},x_{2},\eta))$, where
\begin{equation*}
P_{\pm}: \partial_{\pm}S\mathbb{D} \ni (x,\eta) \mapsto P_{\pm}(x,\eta) = \pm \theta(\eta) - \det(x,\theta(\eta))\theta(\eta)^{\perp} \in \Gamma_{\pm,\eta}.
\end{equation*}

Since the measurements are necessarily discrete, we introduce the fixed parameters $N_{x}$, $N_{d}$ and  $N_{\mathrm{scat}}$ to represent the number of grid points in each spatial dimension, the number of directions measured, and the number of scattering terms computed in the series $T_{1}^{-1}\sum_{m=0}^{\infty}(KT_{1}^{-1})^{m}J$, respectively. Typically, we will take each such parameter to be a power of $2$, since FFT algorithms are used to rotate the grids. Then the basic idea for computing $T_{1}^{-1}$ is to solve equations of the type (\ref{rotatedtransport}) by first computing $g_{\eta}$ and  $\sigma_{\eta}$ by rotating the images of $\sigma(\cdot, \eta)$ and $g(\cdot, \eta)$ clockwise by the angle $\eta$ in the $x$ variable. Then we solve (\ref{rotatedtransport}) for $u_{\eta}$, which can be done, for example, by using a simple first order Euler method along each column of the cartesian grid. Specifically, denote $s_{x} := \frac{2}{N_{x}}$, we set $u_{\eta}(x_{1},:) = 0$ and consider the cartesian grid $\{\mathbf{x}_{1}, \ldots, \mathbf{x}_{N_{x}}\} \otimes \{\mathbf{y}_{1},\ldots, \mathbf{y}_{N_{x}}\}$ where $\mathbf{x}_{i}, \mathbf{y}_{i} = -1+ (i-\frac{1}{2})s_{x}$ for $1 \leq i \leq N_{x}$. Then the Euler method would give the update
\begin{equation*}
u_{\eta}(x_{j},\mathbf{y}) = u_{\eta}(x_{j-1},\mathbf{y}) + s_{x}\left(g_{\eta}(x_{j-1},\mathbf{y}) - \sigma_{\eta}(x_{j-1},\mathbf{y})u_{\eta}(x_{j-1},\mathbf{y}) \right), \qquad  2 \leq j \leq N_{x}.
\end{equation*}

In order to compute $K$, we define the angular step size $\delta := \frac{2\pi}{N_{d}}$ and sum over the set of angles $\{\eta_{1},\ldots, \eta_{N_{d}}\}$ where $\eta_{i} = \left( 1 - \frac{1}{2} \right) \delta$. We can then approximate $K$ by the discrete operator $K_{\delta}$ which has the formula
\begin{equation}
K_{\delta}g(\mathbf{x},\eta_{i}) = \delta \sum_{j=1}^{N_{d}} g(\mathbf{x},\eta_{j})k(\mathbf{x},\eta_{j},\eta_{i}), \qquad  1 \leq i \leq N_{d}. \label{Kdiscrete}
\end{equation}

\subsection{Computing $X_{V}$ \label{computeXV}}
We describe in Algorithm \ref{forwardalg} the iterative scheme to numerically solve (\ref{transport}) and simulate the partial data $X_{V}f$ with respect to some subset $V \subset \partial_{+}S\Omega$.
\begin{algorithm}
\caption{Computing $X_{V}f$}
\label{forwardalg}
\begin{algorithmic}[1]
\State Let $u_{0}$ solve $\theta \cdot \nabla u + \sigma u = f$
\For{$i=1$ to $N_{d}$}
\State \textbf{compute} $\sigma_{\eta_{i}}$ and $f_{\eta_{i}}$
\State \textbf{solve} $\partial_{x_{1}}[u_{0}]_{\eta_{i}} + \sigma_{\eta_{i}}[u_{0}]_{\eta_{i}} = f_{\eta_{i}}$ for $[u_{0}]_{\eta_{i}}$
\State $u_{0}(\mathbf{x}, \mathbf{y}, \eta_{i}) \gets \left[ [u_{0}]_{\eta_{i}} \right]_{-\eta_{i}}(\mathbf{x}, \mathbf{y})$
\EndFor
\For{$j=1$ to $N_{\mathrm{scat}}$}
\State $f_{j} \gets K_{\delta}u_{j-1}$
\For{$i=1$ to $N_{d}$}
\State \textbf{compute} $[f_{j}]_{\eta_{i}}$
\State \textbf{solve} $\partial_{x_{1}}[u_{j}]_{\eta_{i}} + \sigma_{\eta_{i}}[u_{j}]_{\eta_{i}} = [f_{j}]_{\eta_{i}}$ for $[u_{j}]_{\eta_{i}}$
\State $u_{j}(\mathbf{x}, \mathbf{y}, \eta_{i}) \gets \left[ [u_{j}]_{\eta_{i}} \right]_{-\eta_{i}}(\mathbf{x}, \mathbf{y})$
\EndFor
\EndFor
\State $u \gets \sum_{j=0}^{N_{\mathrm{scat}}}u_{j}$
\State \textbf{compute} $\chi_{V}u \Big\vert_{\partial_{+}S\mathbb{D}}$
\end{algorithmic}
\end{algorithm}

What we have not yet described in much detail is the method used to compute the rotations of each function, which we discuss briefly within the remaining portion of this section. For more details however, we refer the reader to \cite{monard}. The general idea is to write the rotation map $r_{\eta}(\mathbf{x},\mathbf{y}) = (\mathbf{x}\cos{\eta} + \mathbf{y}\cos{\eta}, -\mathbf{x}\sin{\eta} + \mathbf{y}\cos{\eta})$ as a composition of dilations and shearing/slant operations in each variable separately. In particular, we can write
\begin{equation*}
\begin{array}{rlll}
r_{\eta} & = & d_{\mathbf{x}, \frac{1}{\cos{\eta}}} \circ s_{\mathbf{x}, -\sin{\eta}} \circ d_{\mathbf{y},\cos{\eta}} \circ s_{\mathbf{y},\tan{\eta}}, &   \\
s_{\mathbf{y}, \alpha}(\mathbf{x},\mathbf{y}) & = & (\mathbf{x}, \mathbf{y} - \alpha \mathbf{x}),  \qquad  s_{\mathbf{x},\beta}(\mathbf{x},\mathbf{y}) = (\mathbf{x}-\beta \mathbf{y},\mathbf{y}), & \alpha,\beta \in \mathbb{R}\\ 
d_{\mathbf{x},t}(\mathbf{x},\mathbf{y}) & = & (t\mathbf{x}, \mathbf{y}), \qquad  d_{\mathbf{y},t}(\mathbf{x},\mathbf{y}) = (\mathbf{x}, t\mathbf{y}), & t \in \mathbb{R}.
\end{array}
\end{equation*}
The shearing/slant operations $s_{\mathbf{x},\beta}$ and $s_{\mathbf{y},\alpha}$ are implemented in phase space using a periodic interpolation function together with some identities utilizing the discrete forward and inverse Fourier transform. Specifically, we first embed the $N_{x} \times N_{x}$ image into a $2N_{x} \times N_{x}$ image, padding the top and bottom of the image with $\frac{N_{x}}{2} \times N_{x}$ arrays of zeros. We then perform the vertical shearing operation $s_{\mathbf{y},\alpha}$, which independently shifts each column of the image by an amount that depends linearly on the column index with factor $-\alpha$. 

The operation of shifting a vector $x = [x_{1},\ldots, x_{m}]^{T}$ by an amount $s$ is done in phase space. First we define the $2m$-periodic function
\begin{equation}
D_{m}(y) = \frac{\sin(\pi y)}{m \sin\left( \frac{\pi y}{m} \right)}, \quad y \in \mathbb{R}.
\end{equation}
We then define the spectral interpolant
\begin{equation}
\widetilde{x}(y) := \sum_{l=1}^{m}x_{l}D_{m}(y-l),
\end{equation}
which coincides with $x_{j}$ when $y = j$ and interpolates between those values sinusoidally (see Figure \ref{spectralinterp}). The spectrally shifted vector $x_{s}$ is given by
\begin{equation*}
x_{s} = [\widetilde{x}(1+s), \widetilde{x}(2+s),\ldots, \widetilde{x}(m + s)]^{T}.
\end{equation*}
Note that $x_{s}$ is not a priori defined for $s \in \mathbb{Z}$, but it can be extended continuously to such points due to the structure of the singularities of  $D_{m}$.
\begin{figure}
\centering
\includegraphics[width=0.8\textwidth, clip=true, trim = 140 50 110 40]{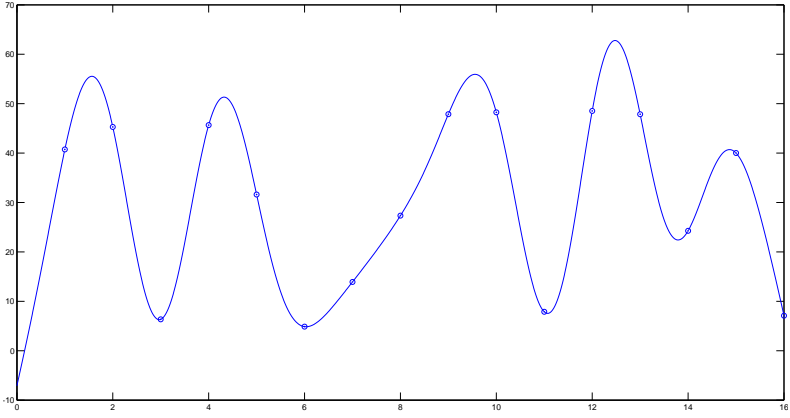}
\caption{The values of a random $16$-vector $x$ are plotted while overlayed with the spectral interpolant $\widetilde{x}(y) = \sum_{l=1}^{m}x_{l}D_{m}(y-l)$ for $y \in [0,16]$.\label{spectralinterp}}
\end{figure}

Now, in practice we have first padded the image above and below with zeros, so we have $m = 2N_{x}$ which is even. It is then straightforward to verify that
\begin{equation*}
D_{2N_{x}}(t) = \frac{\sin(\pi t)}{2N_{x}\sin{ \frac{\pi t}{2N_{x}}}} = \frac{1}{2N_{x}}\sum_{l=-N_{x}}^{N_{x}-1}e^{i \frac{\pi}{N_{x}}(l + \frac{1}{2})t}, \quad t \in [0,2N_{x}].
\end{equation*}
Recall the $N$-point discrete Fourier and inverse Fourier transforms, given by
\begin{align*}
X(k) & = F_{j \mapsto k}^{N}[x(j)] = \sum_{j=1}^{N}x(j)e^{-\frac{2 \pi i}{N}(j-1)(k-1)}, \quad k=1,\ldots, N,\\
x(j) & = F_{k \mapsto j}^{-1, N}[X(k)] = \frac{1}{N}\sum_{k=1}^{N}X(k)e^{2\pi i}{N}(j-1)(k-1), \quad j = 1,\ldots, N.
\end{align*}
It is then possible to write $\widetilde{x}(l-s)$ for $l=1,\ldots, 2N_{x}$ as a composition of discrete Fourier transforms and inverse transforms and multiplication by complex exponentials. In particular,
\begin{align}
& \quad \widetilde{x}(l-s) \\
& = e^{i \frac{\pi}{N_{x}}(-N_{x} + \frac{1}{2})(l-1)}F_{k \mapsto l}^{-1, 2N_{x}}\left[ e^{-i \frac{\pi}{N_{x}}(k-1-N_{x} + \frac{1}{2})s}
F_{j \mapsto k}^{2N_{x}}\left[ x(j)e^{-i \frac{\pi}{N_{x}}(-N_{x} + \frac{1}{2})(j-1)}\right] \right]. \notag
\end{align}

The dilation operations $d_{\mathbf{x},t}$ and $d_{\mathbf{y},t}$ are computed via a resampling done in the Fourier domain. In particular, we must resample a vector $x$ of size $2N_{x}$ down to a vector $\overline{x}$ of size $N_{x}$ but with a different step size. We can do this by using the $N$-point fractional Fourier transform with coefficient $\alpha$, defined by
\begin{equation}
X(l) = G_{k \mapsto l}^{N, \alpha}[x] = \sum_{k=1}^{N}x(k)e^{-2\pi i \alpha (k-1)(l-1)}.
\end{equation}
For example, if we start with a vector $x$ sampled at the gridpoints $\{j-1\}_{j=1}^{2m}$ and we want to shift $x$ by $s$ (i.e. sample at $\{j-1+s\}_{j=1}^{2m}$) and then resample back to a vector $\overline{x}$ taking values at the $m$ points $y_{l} = s+h(l-1)$ for $l=1,\ldots,m$, then
\begin{equation}
\overline{x}(y_{l}) = \frac{1}{2m} e^{i \frac{\pi}{m}(-m + \frac{1}{2})h(l-1)}G_{k \mapsto l}^{2m, \frac{-h}{2m}}\left[ e^{i \frac{\pi}{m}(k - m - 1 + \frac{1}{2})s}F_{j \mapsto k}^{2m}\left[ x(j)e^{-i \frac{\pi}{m}(-m + \frac{1}{2})(j-1)} \right] \right].
\end{equation}
This corresponds to a vertical dilation composed wth a vertical shift (see \cite{fractionalfourier,monard}).

\begin{figure}
\centering
\includegraphics[width=\textwidth, clip=true, trim =145 0 130 10]{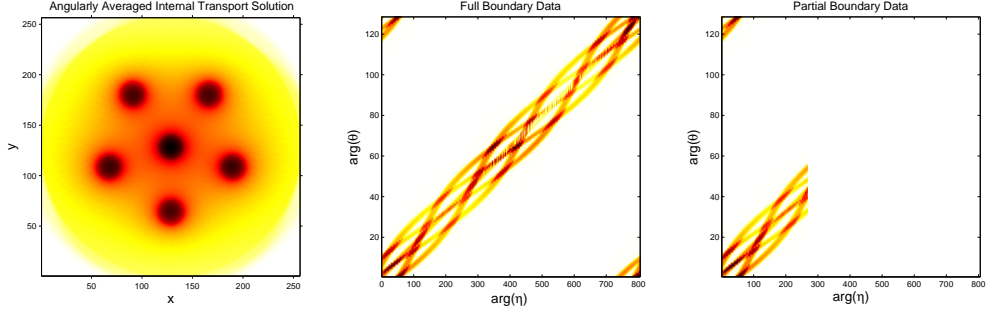}
\caption{On the left is the angularly averaged solution to (\ref{transport}) corresponding to the internal source given in Figure \ref{disksnonoise}. The other two images show the complete and partial data $Xf$ and $X_{V}f$, respectively, where $V$ is the outgoing conic set above the arc of $\partial \mathbb{D}$ defined by $\arg{\eta} \in [0,\pi/3]$. The axes are labeled such that $\theta$ is the transport direction and $\eta$ is the position along the boundary $\partial \mathbb{D}$, with $\arg{\eta} = 0$ corresponding to the point $(1,0)$. \label{solutiondata}}
\end{figure}

\section{Approximating the Normal Operator $X_{V}^{*}X_{V}$ \label{ApproxNormal}}
After approximating the solution to the forward problem and restricting to $\partial_{+}S\Omega$ to obtain the simulated data $Xf$, we can then proceed with computing $X^{*}Xf$, and similarly, $X_{V}^{*}X_{V}f$. Recall that for any $N \in \mathbb{N}$
\begin{align*}
X_{V}f &= \chi_{V}R_{+}T_{1}^{-1}(\mathrm{Id} - KT_{1}^{-1})^{-1}Jf\\
& = \chi_{V}R_{+}T_{1}^{-1}\sum_{m=0}^{N}(KT_{1}^{-1})^{m}Jf + \chi_{V}R_{+}T_{1}^{-1}(KT_{1}^{-1})^{N+1}(\mathrm{I} - KT_{1}^{-1})^{-1}Jf.
\end{align*}
Under the condition that $\|KT_{1}^{-1}\|_{L^{2}\to L^{2}} < 1$, we then have
\begin{equation}
X_{V}^{*}  = J^{*}\left( \sum_{m=0}^{\infty}([T_{1}^{-1}]^{*}K^{*})^{m}\right)[ \chi_{V}R_{+}T_{1}^{-1}]^{*} \label{xvadjoint}
\end{equation}
To numerically compute $X_{V}^{*}$, we separately compute the three types of terms appearing in (\ref{xvadjoint}).

Given that $J:L^{2}(\Omega) \to L^{2}(\Omega \times \mathbb{S}^{n-1})$, it is straightforward to verify that
\begin{equation}
J^{*}g(x) = \int_{\mathbb{S}^{n-1}}g(x,\theta)\,d\theta. \label{jadjoint}
\end{equation}
We also compute
\begin{equation}
K^{*}g(x,\theta) = \int_{\mathbb{S}^{n-1}}k(x,\theta',\theta)g(x,\theta')\,d\theta'. \label{kadjoint}
\end{equation}
In the isotropic scattering case, this gives
\begin{equation}
K^{*}g(x,\theta) = k(x)\int_{\mathbb{S}^{n-1}}g(x,\theta')\,d\theta'.
\end{equation}
We will need the discretized version of (\ref{kadjoint}) in the same vain as (\ref{Kdiscrete}), which is given by
\begin{equation}
K_{\delta}^{*}g(x, \eta_{i}) = \delta \sum_{j=1}^{N_{d}}k(x, \eta_{j}, \eta_{i})g(x, \eta_{j}), \qquad 1\leq i \leq N_{d} \label{kadjointdiscrete}
\end{equation}

One can also verify that $T_{1}^{-1}:L^{2}(\Omega \times \mathbb{S}^{n-1}) \to  L^{2}(\Omega \times \mathbb{S}^{n-1})$ has adjoint
\begin{equation}
[T_{1}^{-1}]^{*}g(x,\theta) =\int_{0}^{\infty}\exp\left( -\int_{0}^{s}\sigma(x + \tau\theta,\theta)\,d\tau \right)g(x +s\theta,\theta)\,ds. \label{t1invadjoint}
\end{equation}
However, in order to apply $[T_{1}^{-1}]^{*}$ numerically, it is easier to use the associated first order differential equation for which it is the solution operator. We already know that $T_{1}^{-1}$ is a left inverse for $\theta \cdot \nabla + \sigma$ with the boundary condition in (\ref{transport}). So naturally $-\theta \cdot \nabla + \sigma$ is the associated differential operator for $[T_{1}^{-1}]^{*}$.

Thus we can use a first order Euler method to compute $[T_{1}^{-1}]^{*}$ just as with $T_{1}^{-1}$. Finally, we make note of a formula for $[\chi_{V}R_{+}T_{1}^{-1}]^{*}:L^{2}(\partial_{+}S\Omega, d\Sigma) \to L^{2}(\Omega \times \mathbb{S}^{n-1})$, which is easily verified to be
\begin{equation}
[\chi_{V}R_{+}T_{1}^{-1}]^{*}g(x,\theta) = g^{\#}(x,\theta)\chi_{V}^{\#}(x,\theta)E(x,\theta).
\end{equation}

Recall the truncation parameter $N_{\mathrm{scat}} \in \mathbb{N}$  corresponding to how many terms in the Neumann series $\sum_{m=0}^{\infty}([T_{1}^{-1}]^{*}K^{*})^{m}$ to use. We proceed as in Algorithm \ref{normalalg}.
\begin{algorithm}
\caption{Computing $X_{V}^{*}X_{V}f$}
\label{normalalg}
\begin{algorithmic}[1]
\State $v_{0}(x, \eta) \gets [\chi_{V}R_{+}T_{1}^{-1}]^{*}X_{V}f(x,\eta) = (X_{V}f)^{\#}(x,\eta)E(x,\eta)$
\For{$j=1$ to $N_{\mathrm{scat}}$}
\State $v_{j} \gets K_{\delta}^{*}v_{j-1}$
\Comment{Apply $K^{*}$}
\For{$i=1$ to $N_{d}$}
\State $w_{\eta_{i}}(\mathbf{x}_{N_{x}}, \mathbf{y}) = 0$
\Comment{Apply $[T^{-1}]^{*}$}
\State \textbf{solve} $-\partial_{x_{1}}w_{\eta_{i}} + \sigma_{\eta_{i}} w_{\eta_{i}} = [v_{j}]_{\eta_{i}}$ for $w_{\eta_{i}}$.
\State $v_{j}(\mathbf{x},\mathbf{y},\eta_{i}) \gets w = \left[ w_{\eta_{i}} \right]_{-\eta_{i}}(\mathbf{x},\mathbf{y})$
\EndFor
\EndFor
\State $v(\mathbf{x}, \mathbf{y}) \gets \delta \sum_{i=1}^{N_{d}}\sum_{j=1}^{N_{\mathrm{scat}}}v_{j}(\mathbf{x},\mathbf{y},\eta_{i})$
\Comment{Apply $J^{*}$}\\
\Comment{$v$ is an approximation to $X_{V}^{*}X_{V}f$}
\end{algorithmic}
\end{algorithm}

\section{Smoothness Analysis \label{smoothnessanalysis}}
When trying to recover microlocally the most singular part of the source via $X_{V}^{*}X_{V}$ as analyzed in \cite{hubenthal}, it turns out that in theory only the first term of the Neumann series is needed. This is stated in Lemma \ref{approxnormalsmooth}, for which we now provide the proof, utilizing the weakly singular integral results in Appendix \ref{ap:compsing}.
\begin{proof}[Proof of Lemma \ref{approxnormalsmooth}] Observe that
\begin{align*}
X_{V}^{*} & = [ \chi_{V}R_{+}T_{1}^{-1}(\mathrm{I} - KT_{1}^{-1})^{-1}J]^{*} \\
& = \sum_{j=0}^{m_{2}}[\chi_{V}R_{+}T_{1}^{-1}(KT_{1}^{-1})^{j}J]^{*} + [\chi_{V}R_{+}T_{1}^{-1}(KT_{1}^{-1})^{m_{2}+1}(\mathrm{I} - KT_{1}^{-1})^{-1}J]^{*}\\
X_{V} & = \chi_{V}R_{+}T_{1}^{-1}(\mathrm{I} - KT_{1}^{-1})^{-1}J\\
& = \sum_{i=0}^{m_{1}}\chi_{V}R_{+}T_{1}^{-1}(KT_{1}^{-1})^{i}J + \chi_{V}R_{+}T_{1}^{-1}(KT_{1}^{-1})^{m_{1}+1}(\mathrm{I} - KT_{1}^{-1})^{-1}J.
\end{align*}
Thus
\begin{align*}
& X_{V}^{*}X_{V}f - \sum_{j=0}^{m_{2}}\sum_{i=0}^{m_{1}}[\chi_{V}R_{+}T_{1}^{-1}(KT_{1}^{-1})^{j}J]^{*}\chi_{V}R_{+}T_{1}^{-1}(KT_{1}^{-1})^{i}Jf\\
& =\sum_{i=0}^{m_{1}}[\chi_{V}R_{+}T_{1}^{-1}(\mathrm{I} - KT_{1}^{-1})^{-1}(KT_{1}^{-1})^{m_{2}+1}J]^{*}\chi_{V}R_{+}T_{1}^{-1}(KT_{1}^{-1})^{i}Jf\\
& +  \sum_{j=0}^{m_{2}}[\chi_{V}R_{+}T_{1}^{-1}(KT_{1}^{-1})^{j}J]^{*}\chi_{V}R_{+}T_{1}^{-1}(KT_{1}^{-1})^{m_{1}+1}(\mathrm{I} - KT_{1}^{-1})^{-1}Jf\\
& + [\chi_{V}R_{+}T_{1}^{-1}(\mathrm{I} - KT_{1}^{-1})^{-1}(KT_{1}^{-1})^{m_{2}+1}J]^{*}\chi_{V}R_{+}T_{1}^{-1}(KT_{1}^{-1})^{m_{1}+1}(\mathrm{I} - KT_{1}^{-1})^{-1}Jf\\
& =: A_{1} + A_{2} + A_{3}.
\end{align*}

First note that 
\begin{align*}
& \quad [\chi_{V}R_{+}T_{1}^{-1}]^{*}\chi_{V}R_{+}T_{1}^{-1}f(x,\theta)\\
& = \left|\chi_{V}(x,\theta)\right|^{2}\int \exp\left( -\int_{s}^{\infty}\sigma(x + \tau \theta, \theta) \,d\tau \right)f(x+s\theta, \theta)\,ds.
\end{align*}
In particular, $[\chi_{V}R_{+}T_{1}^{-1}]^{*}\chi_{V}R_{+}T_{1}^{-1}$ is bounded on $H^{j}(\Omega \times \mathbb{S}^{n-1})$ for all $j$. Moreover, from the proof of Proposition 6.1 of \cite{hubenthal} we have that $[( \mathrm{Id} - KT_{1}^{-1})^{-1}]^{*}$ can be taken to be bounded on $H^{j}(\Omega \times \mathbb{S}^{n-1})$ whenever $(\mathrm{Id} - KT_{1}^{-1})^{-1}$ is, which we can assume is the case for $j \leq l + \max\{m_{1}, m_{2}\} + 1$. From the discussion in Appendix \ref{ap:compsing}, we have that
\begin{equation*}
[(KT_{1}^{-1})^{m_{2}+1}J]^{*}g(x) = \iint \frac{\beta_{m_{2} +1}(y,x,|y-x|,\widehat{y-x},\theta)}{|y-x|^{n-m}}g(y,\theta)\,dy\,d\theta.
\end{equation*}
Using a similar argument as in Lemma 2 of \cite{inversesource} and applying Proposition \ref{pseudoiterate}, we have that $[(KT_{1}^{-1})^{m_{2}+1}J]^{*}$ is bounded from $H^{l}(\Omega \times \mathbb{S}^{n-1}) \to H^{l+m_{2}+1}(\Omega \times \mathbb{S}^{n-1})$. 

Now, to analyze $A_{1}$ we write
\begin{equation*}
A_{1} = [(KT_{1}^{-1})^{m_{2}+1}J]^{*}[ (\mathrm{Id} - KT_{1}^{-1})^{-1}]^{*}[\chi_{V}R_{+}T_{1}^{-1}]^{*}\chi_{V}R_{+}T_{1}^{-1}\sum_{i=0}^{m_{1}}(KT_{1}^{-1})^{i}J,
\end{equation*}
and from the arguments of the previous paragraph it is clear that $A_{1}$ maps $H^{l}(\Omega)$ to $H^{l+m_{2}+1}(\Omega_{1})$. Likewise,
\begin{equation*}
A_{2} = \sum_{j=0}^{m_{2}}[KT_{1}^{-1})^{j}J]^{*}[\chi_{V}R_{+}T_{1}^{-1}]^{*}\chi_{V}R_{+}T_{1}^{-1}(\mathrm{I} - KT_{1}^{-1})^{-1}(KT_{1}^{-1})^{m_{1}+1}J
\end{equation*}
and so $A_{2}$ maps $H^{l}(\Omega)$ to $H^{l+m_{1}+1}(\Omega_{1})$. A similar argument shows that $A_{3}$ maps $H^{l}(\Omega)$ to $H^{l+m_{1}+m_{2}+2}(\Omega_{1})$.
\qquad \end{proof}

\section{Numerical Computations \label{examples}}
For our numerical computations, the goal is to provide visual verification of Theorem \ref{microtheorem}, and more specifically, of the related result of Lemma \ref{approxnormalsmooth}. In all examples, we use a cartesian grid of 256 by 256 with 128 directions $\theta$. We'll use the notation $(x,y)$ to denote a point in $\rn{2}$. In order to incorporate anisotropy in $k$, we use the Henyey-Greenstein (H-G) phase function \cite{henyey}, which is very commonly used in optical imaging:
\begin{equation}
p(\theta,\theta') = \left\{ \begin{array}{cc} \frac{1}{2\pi} \frac{1-g^{2}}{1 + g^{2} - 2g \theta \cdot \theta'} & n=2,\\
 \frac{1}{4\pi} \frac{1-g^2}{(1 + g^{2} - 2g\theta \cdot \theta')^{3/2}} & n=3.
\end{array}\right.
\end{equation}
To set $g =0$ corresponds to the isotropic case, while $g = 1$ and $g = -1$ correspond to forward scattering and backscattering, respectively. In typical applications $g$ is around $0.80$ to $0.95$ (\cite{hongkai}), which is characteristic of highly forward-peaked scattering. 

We also incorporate noise into the boundary data in the following way. The noiseless data $Xf$ is an $m \times 128$ matrix, with the rows corresponding to the position along the unit circle and the columns corresponding to the particles' direction. Given a parameter $\mu > 0$, for the $j$th column $v_{j}$ of $Xf$ (i.e. $[Xf]_{j}$) we compute $\|v_{j}\|_{2} = \sqrt{v_{j}^{T}v_{j}}$ and generate a vector $w$ of length $m$ with values chosen randomly according to the standard normal distribution (variance $1$ and mean $0$). We then define the $j$th column of the noisy data by
\begin{equation}
[Xf_{\mathrm{noise}}]_{j} := [Xf]_{j} + \mu \sqrt{[Xf]_{j}^{T}[Xf]_{j}}\frac{w}{\sqrt{w^{T}w}} = v_{j} + \mu \|v_{j}\|_{2} \frac{w}{\|w\|_{2}}.
\end{equation}
Thus $\mu$ represents the fraction of $\|v_{j}\|_{2}$ to which we would like to rescale the variance of the noise.

With these details in mind, for all computed examples we have taken
\begin{equation*}
k(x,y,\theta,\theta') = \frac{1}{2\pi} \chi_{\{x^2 + y^2 < 1\}}(x,y)\left(0.05 + \sin^{2}{xy}\right)\frac{1-0.85^2}{1 + 0.85^2 - 2\cdot 0.85 \cdot \theta \cdot \theta' }
\end{equation*}
and
\begin{equation*}
\sigma(x,y,\theta) = 0.5 \chi_{\{x^{2} + y^{2} < 1\}}(x,y)[0.05 + \cos^{2}{xy}]\sin^{2}{\theta}.
\end{equation*}
Moreover, using the notation of Lemma \ref{approxnormalsmooth}, we take $m_{1} = 8$ and $m_{2} = 2$. This corresponds to computing the first 9 terms of the Neumann series expansion for $Xf$ and only $3$ terms in the series representation of $X^{*}$. Finally, in all computations we take
\begin{equation}
V = V_{\mathrm{examples}} = \{(\eta,\theta) \in \partial_{+}S\Omega = \partial \mathbb{D} \times \mathbb{S}^{1} \, | \, \arg{\eta} \in [0,\pi/3] \textrm{ and } \mathbf{\eta} \cdot \mathbf{\theta} > 0\}. \label{Vexamples}
\end{equation}

\begin{figure}
\centering
\includegraphics[width=\textwidth, trim = 157 174 110 135, clip=true]{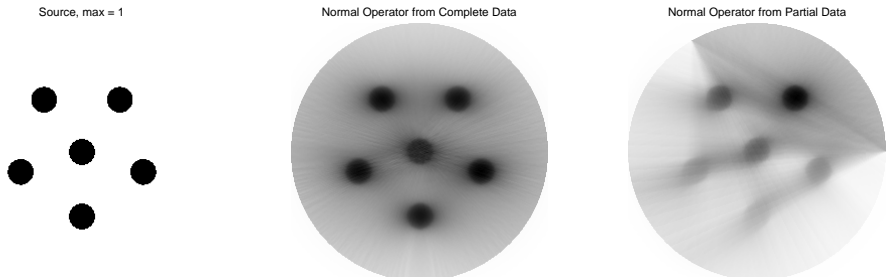}
\caption{A source consisting of circular bump functions with height $1$. The partial data is measured without noise on the set $V = V_{\mathrm{examples}}$.}\label{disksnonoise}
\end{figure}
\begin{figure}
\centering
\includegraphics[width=\textwidth, trim = 157 174 110 135, clip=true]{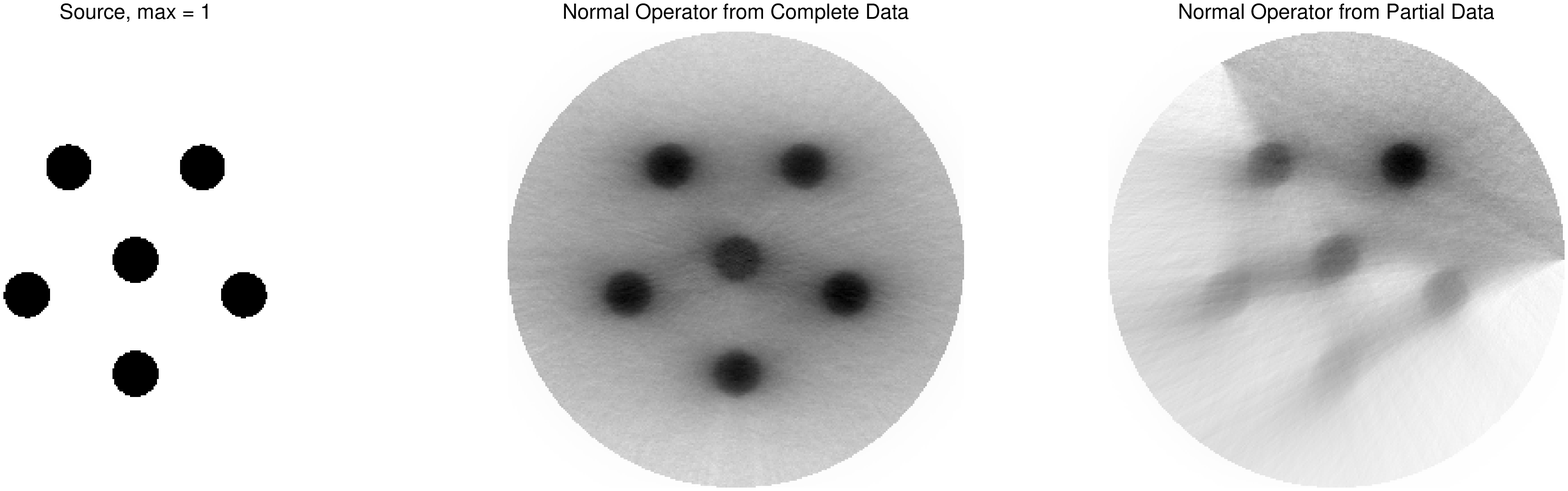}
\caption{Same as Figure \ref{disksnonoise} except with a noise coefficient $\mu = 0.50$.}\label{disksnoise}
\end{figure}
\begin{figure}
\centering
\includegraphics[width=\textwidth, trim =157 174 110 135, clip=true]{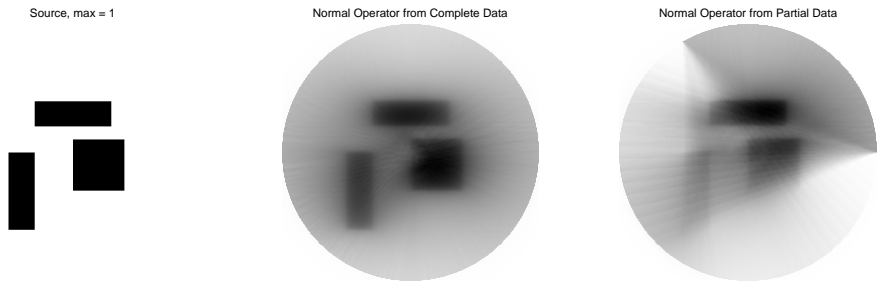}
\caption{A source consisting of rectangular bump functions with height $1$. The partial data is measured without added noise on the same set $V$ as given in Figure \ref{disksnonoise}.}\label{rectangles}
\end{figure}
\begin{figure}
\centering
\includegraphics[width=\textwidth, trim = 157 174 110 135, clip=true]{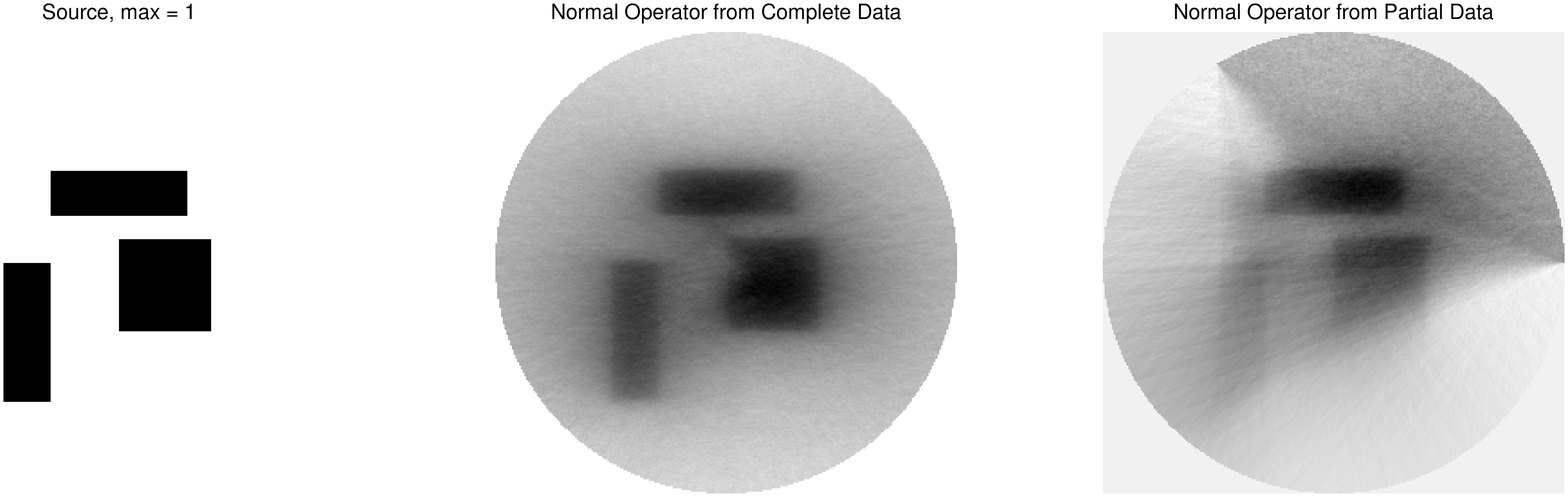}
\caption{Same as Figure \ref{rectangles} except with a noise coefficient $\mu = 0.5$.}\label{rectanglesnoise}
\end{figure}
\begin{figure}
\centering
\includegraphics[width=\textwidth, trim = 157 174 110 135, clip=true]{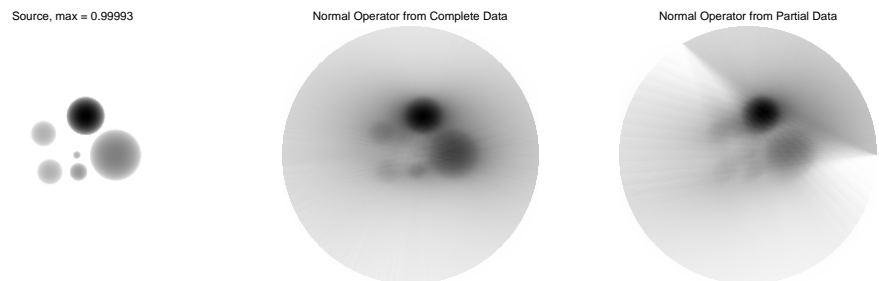}
\caption{A spiral pattern of continuous circular bump functions of the form $g(x,y) = A\sqrt{1 - \frac{1}{r^{2}}(x-x_{0})^{2} - \frac{1}{r_{0}^{2}}(y-y_{0})^{2}}$, where $A$ is the maximum height and $r$ is the radius. In this example, starting from the widest bump and moving counterclockwise, we have the set of heights and radii $A = \{0.5, 1, 0.3, 0.3, 0.4, 0.3\}$ and $r = \{ 0.2, 0.15, 0.1, 0.1, 0.07, 0.03\}$, respectively. The partial data is measured with no added noise on the set $V$ given in Figure \ref{disksnonoise}.}\label{spiral}
\end{figure}
\begin{figure}
\centering
\includegraphics[width=\textwidth, trim = 157 174 110 135, clip=true]{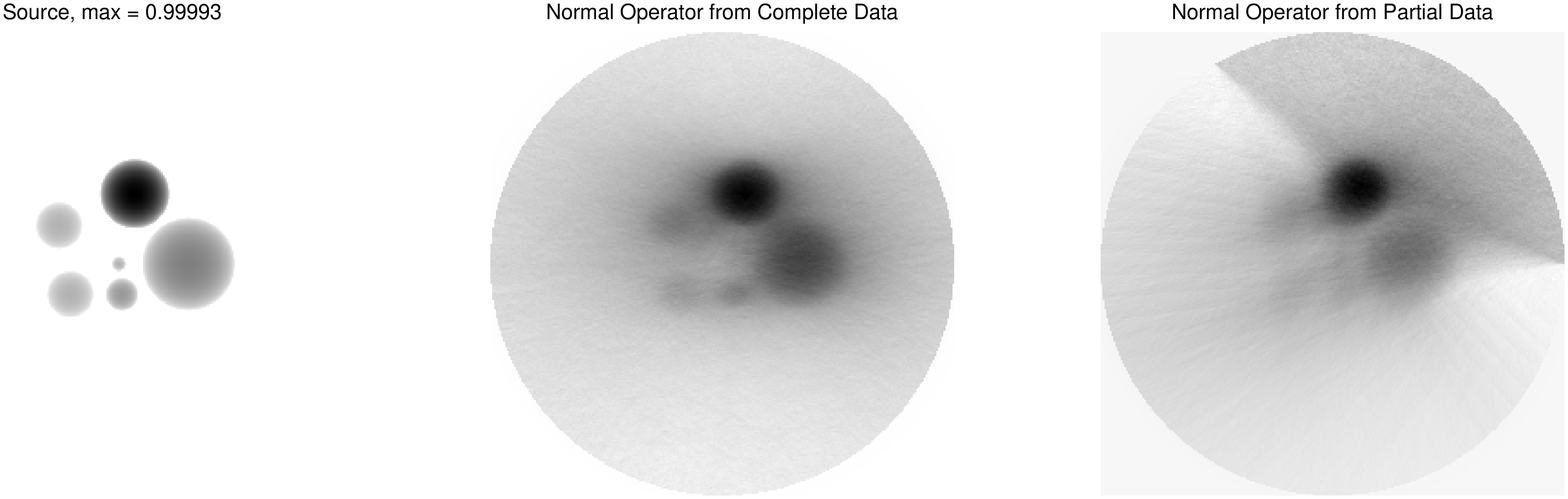}
\caption{Same as Figure \ref{spiral}, but with an added noise coefficient $\mu = 0.50$.}\label{spiralnoise}
\end{figure}
\section{Conclusions}
We have presented a numerical method to solve the direct problem for the Radiative Transfer Equation (\ref{transport}) based on the work of \cite{monard}, which utilizes the discrete Fourier transform and fractional discrete Fourier transform to implement a rotation-based method for computing the necessary line integrals. Moreover, we have computed the normal operator $X_{V}^{*}X_{V}$ in a similar manner, where $X_{V}$ is the partial data operator. Ultimately, this has given us some nice visual examples in the presence of anisotropic $(\sigma, k)$ with or without added noise, where the anisotropic parts of the parameters have physically meaningful structure as given in \cite{hongkai, henyey}. Such examples visually demonstrate the ability to detect the singularities of an unknown source with only partial data of the transport solution $u$ at the boundary.

Most importantly, the results of Figures \ref{disksnonoise} through \ref{spiralnoise} show the detection of edges resulting from singularities in the respective chosen sources $f$, and this is consistent with what is theoretically predicted in Theorem \ref{microtheorem} and with the microlocally visible set $\mathcal{M}'$ given in (\ref{microvisibleset}).  Moreover, it provides a nice generalization of similar well-known results regarding the detectable singularities in the context of limited data computed X-ray tomography, see \cite{quinto}.

\Appendix \section{Smoothing Properties of Compositions of Weakly Singular-type Integral Operators \label{ap:compsing}}
Let's consider an operator of type
\begin{equation}
[Af](x,\theta) = \int \frac{\alpha(x,y,|x-y|, \widehat{x-y},\theta)}{|x-y|^{n-1}}f\left( y, \widehat{x-y} \right)\,dy
\end{equation}
where $\alpha(x,y,r,\eta,\theta) \in C_{c}^{\infty}(\rn{n} \times \rn{n} \times \mathbb{R}_{+} \times \mathbb{S}^{n-1} \times \mathbb{S}^{n-1})$. We would like to analyze compositions of the form $[A^{m}Jf](x,\theta)$ where $J:L^{2}(\rn{n}) \to L^{2}(\rn{n} \times \mathbb{S}^{n-1})$ is the extension operator $Jf(x,\theta) = f(x)$ as used previously.

For $m = 2$, we compute
\begin{align*}
& \quad [A^{2}Jf](x,\theta)\\
& = \int \frac{\alpha(x,y_{1},|x-y_{1}|,\widehat{x-y_{1}},\theta)}{|x-y_{1}|^{n-1}}[AJf]\left(y_{1}, \frac{x-y_{1}}{|x-y_{1}|}\right)\,dy_{1}\\
& = \int \frac{\alpha(x,y_{1},|x-y_{1}|,\widehat{x-y_{1}},\theta)}{|x-y_{1}|^{n-1}} \int \frac{\alpha\left(y_{1},y_{2},|y_{1}-y_{2}|,\widehat{y_{1}-y_{2}}, \widehat{x-y_{1}}\right)}{|y_{1}-y_{2}|^{n-1}}f(y_{2})\,dy_{2}\,dy_{1}\\
& = \int \left( \int \frac{\alpha(x,y_{1},|x-y_{1}|,\widehat{x-y_{1}},\theta)\alpha\left(y_{1},y_{2},|y_{1}-y_{2}|,\widehat{y_{1}-y_{2}}, \widehat{x-y_{1}}\right)}{|x-y_{1}|^{n-1}|y_{1}-y_{2}|^{n-1}} \,dy_{1} \right)\\
& \quad \cdot  f(y_{2})\,dy_{2}.
\end{align*}
Similarly for $m=3$,
\begin{align*}
& \quad [A^{3}Jf](x,\theta)\\
& =  \int \Bigg( \iint \frac{\alpha(x,y_{1},|x-y_{1}|, \widehat{x-y_{1}},\theta)\alpha\left(y_{1},y_{2},|y_{1}-y_{2}|, \widehat{y_{1}-y_{2}}, \widehat{x-y_{1}}\right)}{|x-y_{1}|^{n-1}|y_{1}-y_{2}|^{n-1}|y_{2}-y_{3}|^{n-1}}\\
& \qquad \cdot \alpha\left(y_{2},y_{3}, |y_{2}-y_{3}|, \widehat{y_{2}-y_{3}}, \widehat{y_{1}-y_{2}}\right) \,dy_{2}\,dy_{1}\Bigg) f(y_{3})\,dy_{3}.
\end{align*}
We set $\alpha_{1} := \alpha$, and for $m \geq 2$ define
\begin{align}
& \quad \alpha_{m}(x,y_{1},\ldots, y_{m}, |x-y_{1}|, \ldots, |y_{m-1} - y_{m}|, \widehat{x-y_{1}}, \ldots, \widehat{y_{m-1}-y_{m}}, \theta)  \label{alpham}\\
& :=  \alpha_{m-1}(y_{1}, y_{2}, \ldots, y_{m}, |y_{1}-y_{2}|, \ldots, |y_{m-1}-y_{m}|, \widehat{y_{1}-y_{2}}, \ldots, \widehat{y_{m-1}-y_{m}}, \widehat{x-y_{1}})  \notag \\
& \qquad \cdot \alpha_{1}(x,y_{1},|x-y_{1}|,\widehat{x-y_{1}}, \theta).  \notag
\end{align}
For simplicity of notation, define 
\begin{align*}
\mathbf{y}_{m} & := (y_{1},\ldots, y_{m-1},y),\\
r_{m} & := (|y_{1}-y_{2}|, \ldots, |y_{m-1}-y|),\\
\widehat{\mathbf{r}_{m}} & := (\widehat{y_{1}-y_{2}}, \ldots, \widehat{y_{m-1}-y}).
\end{align*}
Then set for $m \geq 2$
\begin{equation*}
a_{m}(x,y,\theta) := \int \cdots \int \frac{\alpha_{m}(x,\mathbf{y}_{m},|x-y_{1}|, r_{m}, \widehat{x-y_{1}}, \widehat{\mathbf{r}_{m}}, \theta)}{|x-y_{1}|^{n-1}|y_{1}-y_{2}|^{n-1}\cdots |y_{m-1} - y|^{n-1}}\,dy_{1}\cdots dy_{m-1}.
\end{equation*}
Also let
\begin{equation*}
a_{0} := J[\delta(\cdot- y)], \quad a_{1}(x,y,\theta) := \frac{\alpha_{1}(x,y,|x-y|, \widehat{x-y},\theta)}{|x-y|^{n-1}}.
\end{equation*}
Then
\begin{equation*}
[A^{m}Jf](x,\theta) = \int a_{m}(x,y,\theta)f(y)\,dy, \qquad m \geq 0.
\end{equation*}

Now we would like to show that $A^{m}J$ has a kernel of the form
\begin{equation}
\frac{\beta_{m}(x,y,|x-y|, \widehat{x-y}, \theta)}{|x-y|^{n-m}} \label{pseudokernel}
\end{equation}
where $\beta_{m}(x,y,r,\eta,\theta)$ is $C^{\infty}$. Ultimately, our goal is to rigorously show for $m \geq 2$ that $A^{m}J$ is a pseudodifferential operator of order $-m$, by adapting the proof of Lemma 2 in \cite{xraygeneric}, which already directly applies to the case $m=1$.

\subsection{The Case $m=2$}
For $\eta \in \mathbb{S}^{n-1}$, define the set 
\begin{equation*}
\mathcal{D}_{\eta} := \rn{n} \setminus \left(B_{0}\left( \frac{1}{4} \right) \cup B_{\eta}\left( \frac{1}{4} \right)\right).
\end{equation*}
Also let $A_{\eta}$ be the rotational matrix such that $A_{\eta}\eta = e_{1}$, the first unit basis vector of $\rn{n}$. Let $\psi(r)$ be a smooth, even bump function such that $0 \leq \psi \leq 1$, $\psi \equiv 1$ for $|r|\leq \frac{1}{4}$, and $\psi \equiv 0$ for $|r| \geq \frac{1}{2}$. We will frequently use the notation $r = |x-y|$ and $\eta = \widehat{x-y}$. Note that differentiation of $A_{\eta}$ with respect to $x$ yields, from the fact that $\partial_{x_{j}}\widehat{x-y} = \frac{x_{j}-y_{j}}{|x-y|^{2}}$, that $|\partial_{x_{j}}A_{\eta}| \leq \frac{C}{|x-y|}$. Similarly, $l$th order derivatives of $A_{\eta}$ with respect to $x$ will have bounds of the form $\frac{C}{|x-y|^{l}}$. Of course, if we treat $A_{\eta}$ as a matrix of functions of $\eta$ only, and not implicitly dependent on $x$, then each such coordinate function is in $C^{\infty}(\mathbb{S}^{n-1})$.

We have
\begin{equation*}
a_{2}(x,y,\theta) = \int \frac{\alpha_{2}(x,y_{1},y,|x-y_{1}|,|y_{1}-y|, \widehat{x-y_{1}}, \widehat{y_{1}-y}, \theta)}{|x-y_{1}|^{n-1} |y_{1} - y|^{n-1}}\,dy_{1}.
\end{equation*}
We may then cut up the above integral, using $\psi$ and suppressing some variables, to obtain the decomposition
\begin{align*}
I_{1} + I_{2} + I_{3} & := \int \psi\left( \frac{|x-y_{1}|}{|x-y|} \right)\alpha_{2} \,dy_{1} + \int \psi\left( \frac{|y_{1}-y|}{|x-y|} \right)\alpha_{2}\,dy_{1}\\
& \quad + \int \left[ 1 - \psi\left( \frac{|x-y_{1}|}{|x-y|} \right) - \psi\left( \frac{|y_{1}-y|}{|x-y|} \right)\right]\alpha_{2}\,dy_{1}.
\end{align*}
Let's focus on the first term, $I_{1}$. Substitute $w = \frac{x-y_{1}}{|x-y|}$ so that $|x-y|^{n}dw = dy_{1}$, $y_{1} = x - |x-y|w$ and $y_{1} - y = |x-y|( \widehat{x-y} - w)$. We also note that the region of integration is $w \in B_{0}\left( \frac{1}{2} \right)$. We obtain
\begin{align}
I_{1} & = \int \frac{\psi\left( \frac{|x-y_{1}|}{|x-y|} \right)\alpha_{2}(x,y_{1},y,|x-y_{1}|, |y_{1}-y|,\widehat{x-y_{1}},\widehat{y_{1}-y}, \theta)}{|x-y_{1}|^{n-1}|y_{1} - y|^{n-1}}\,dy_{1}\notag\\
& = \frac{1}{|x-y|^{n-2}}\int \frac{\psi(|w|)}{|w|^{n-1}| \widehat{x-y} - w|^{n-1}}\notag\\
& \quad \cdot \alpha_{2}\left(x,x-|x-y|w, y, |x-y||w|, |x-y|| \widehat{x-y} - w|, \widehat{w}, \frac{\widehat{x-y}-w}{|\widehat{x-y}-w|},\theta\right)\,dw\notag\\
& =  \frac{1}{r^{n-2}}\int_{B_{0}\left( \frac{1}{2} \right)} \frac{\psi(|w|)}{|w|^{n-1}| \eta - w|^{n-1}}\notag\\
& \quad \cdot \alpha_{2}\left(x,x-|x-y|w, y, |x-y||w|, |x-y|| \eta - w|, \widehat{w}, \widehat{\eta-w},\theta\right)\,dw\notag\\
& = \frac{1}{r^{n-2}}\int_{B_{0}\left( \frac{1}{2} \right)} \frac{\psi(|w|)\alpha_{2}\left(x,x- rA_{\eta}^{-1}w, y, r|w|, r| e_{1} - w|, \widehat{w}, A_{\eta}^{-1}\widehat{e_{1}-w},\theta\right)}{|w|^{n-1}| e_{1} - w|^{n-1}}\,dw.
\end{align}

Let's consider the other terms:
\begin{align}
I_{2} & = \int \frac{\psi\left( \frac{|y_{1}-y|}{|x-y|} \right)\alpha_{2}(x,y_{1},y,|x-y_{1}|, |y_{1}-y|,\widehat{x-y_{1}},\widehat{y_{1}-y}, \theta)}{|x-y_{1}|^{n-1}|y_{1} - y|^{n-1}}\,dy_{1}\notag\\
& = \frac{1}{r^{n-2}}\int_{B_{0}\left( \frac{1}{2} \right)} \frac{\psi(|\widehat{x-y} - w|)}{|w|^{n-1}| e_{1} - w|^{n-1}}\\
& \quad \cdot \alpha_{2}\left(x,x- rA_{\eta}^{-1}w, y, r|w|, r| e_{1} - w|, \widehat{w}, A_{\eta}^{-1}\widehat{e_{1}-w},\theta\right)\,dw \notag
\end{align}
and
\begin{align}
I_{3} & = \int \frac{\left(1 - \psi\left( \frac{|x-y_{1}|}{|x-y|}\right) - \psi\left( \frac{|y_{1}-y|}{|x-y|} \right)\right)}{|x-y_{1}|^{n-1}|y_{1} - y|^{n-1}}\notag\\
& \qquad \cdot  \alpha_{2}(x,y_{1},y,|x-y_{1}|, |y_{1}-y|,\widehat{x-y_{1}},\widehat{y_{1}-y}, \theta)\,dy_{1}\notag\\
& = \frac{1}{r^{n-2}} \int_{\mathcal{D}_{e_{1}}}\frac{\left(1 - \psi(|w|) - \psi( |e_{1} - w|) \right)}{|w|^{n-1}|e_{1} - w|^{n-1}}\\
& \qquad \cdot \alpha_{2}\left(x,x- rA_{\eta}^{-1}w, y, r|w|, r| e_{1} - w|, \widehat{w}, A_{\eta}^{-1}\widehat{e_{1}-w},\theta\right)\,dw.\notag
\end{align}
Notice that after multiplying each term by $r^{n-2}$, the remaining integrals are smooth in the variables $x,y,r,\eta$ and $\theta$.

\subsection{The General Case}
We seek to record a general formula for the integral representation of $A^{m}Jf(x)$ which resembles a weakly singular integral with integral singularity $\frac{1}{|x-y|^{n-m}}$. Let's make some new definitions to simplify the notation. Set $y_{0} = x$ and $y_{m} = y_{m+1} = y$, and given $1 \leq j \leq m-1$, define 
\begin{equation}
w_{j} := \frac{x-y_{j}}{|x-y_{j+1}|}, \quad w_{j}' := \frac{y_{j}-y_{j+1}}{|x - y_{j+1}|} .
\end{equation}
Note that $\eta = \widehat{x-y} = w_{m}$. Also define $w_{0} = 0$ for convenience, and then set
\begin{align}
\psi_{j}^{1}(w_{j},w_{j}') & := \psi(|w_{j}|)\notag\\
\psi_{j}^{2}(w_{j},w_{j}') & := \psi(|w_{j}'|)\notag\\
\psi_{j}^{3}(w_{j},w_{j}') & := 1 - \psi(|w_{j}|) - \psi(|w_{j}'|).
\end{align}
We notice a few useful formulas for $1 \leq j \leq m-1$:
\begin{align*}
w_{j}' & = \widehat{w_{j+1}} - w_{j},\\
y_{j} & = x - rw_{j}\prod_{l=j+1}^{m-1}|w_{l}|,\\
y_{j} - y_{j+1} & = |x-y_{j+1}|(\widehat{w_{j+1}} - w_{j}).
\end{align*}
Define
\begin{align}
\mathbf{y} & := [y_{1},\ldots, y_{m-1}, y] = \left[ \left\{  x - rw_{j}\prod_{l=j+1}^{m-1}|w_{l}| \right\}_{j=1}^{m-1}, y \right] \in \rn{m},\notag\\
|\mathbf{w}| & :=  \left[ \prod_{l=j}^{m-1}|w_{l}| \right]_{j=1}^{m-1}.
\end{align}
We can partition the integrations involved in the definition of $a_{m}(x,y,\theta)$ as
\begin{align*}
& \int \sum_{\gamma \in \{1,2,3\}^{m-1}} \left[\prod_{j=1}^{m-1}\psi_{j}^{\gamma_{j}}(w_{j},\widehat{w_{j+1}} - w_{j}) \right]\\
& \quad \cdot \frac{\alpha_{m}\left(x,\mathbf{y}, r|\mathbf{w}|, \widehat{w_{1}}, \frac{\widehat{w_{2}} - w_{1}}{|\widehat{w_{2}} -w_{1}|}, \ldots, \frac{\eta - w_{m-1}}{|\eta - w_{m-1}|}, \theta\right)}{r^{n-m}|\eta - w_{m-1}|^{n-1} \prod_{l=1}^{m-1}|w_{l}|^{n-l-1}\prod_{k=1}^{m-2}|\widehat{w_{k+1}} - w_{k}|^{n-1}}\,dw_{1}\ldots \, dw_{m-1}.
\end{align*}
For $0 \leq j \leq m-1$ define
\begin{align}
\mathcal{D}_{j}^{1} & := B_{0}\left( \frac{1}{2} \right)\notag\\
\mathcal{D}_{j}^{2} & := B_{\widehat{w_{j+1}}}\left( \frac{1}{2} \right)\notag\\
\mathcal{D}_{j}^{3} & := \rn{n} \setminus \left( B_{0}\left( \frac{1}{4} \right) \cup B_{\widehat{w_{j+1}}}\left( \frac{1}{4}\right)\right).
\end{align}
For each $\gamma \in \{1,2,3\}_{j=1}^{m-1}$, set $\mathcal{D}_{\gamma} := \prod_{j=1}^{m-1}\mathcal{D}_{j}^{\gamma_{j}}$. Then we have the integral
\begin{align}
& I_{m}(x,y,r,\eta,\theta) :=  \sum_{\gamma \in \{1,2,3\}^{m-1}} \int_{\mathcal{D}_{\gamma}}\left[\prod_{j=1}^{m-1}\psi_{j}^{\gamma_{j}}(w_{j},\widehat{w_{j+1}} - w_{j}) \right] \notag\\
& \quad \cdot \frac{\alpha_{m}\left(x,\mathbf{y}, r|\mathbf{w}|, \widehat{w_{1}}, \frac{\widehat{w_{2}} - w_{1}}{|\widehat{w_{2}} -w_{1}|}, \ldots, \frac{\eta - w_{m-1}}{|\eta - w_{m-1}|}, \theta\right)}{r^{n-m}|\eta - w_{m-1}|^{n-1} \prod_{l=1}^{m-1}|w_{l}|^{n-l-1}\prod_{k=1}^{m-2}|\widehat{w_{k+1}} - w_{k}|^{n-1}}\,dw_{1}\,\ldots \, dw_{m-1}. \label{Amkernel}
\end{align}
Note that $I_{m}(x,y,r,\eta,\theta) \vert_{r = |x-y|, \eta = \widehat{x-y}} = a_{m}(x,y,\theta)$. Assuming that $I_{m}$ is a well-defined, convergent iterated integral, since $\alpha_{m}$ depends smoothly on $r$, it follows that $I_{m}$ is smooth in $r$ upon multiplication by $r^{n-m}$. Also note that since $x$ and $y$ both appear separately in the evaluation of $\alpha_{m}$, $\alpha_{m}$ is compactly supported in $x$, and we will ultimately be applying the given integral operator to functions $f$ supported in a compact set, it follows that multiplying the integral kernel by a smooth cutoff function of $r$ will not change its behavior. Thus we may assume that $I_{m}$ is compactly supported in $r$. Smoothness in $x$, $y$ and $\theta$ is also clear, though we delay consideration of smoothness in $\eta$ until later. We define
\begin{equation}
\beta_{m}(x,y,r,\eta,\theta) := r^{n-m}I_{m}(x,y,r,\eta,\theta),
\end{equation}
so that
\begin{equation*}
A^{m}Jf(x) = \int \frac{\beta_{m}(x,y,|x-y|,\widehat{x-y},\theta)}{|x-y|^{n-m}}f(y)\,dy.
\end{equation*}

In order to justify the convergence of the iterated integral in (\ref{Amkernel}), consider a specific integral summand corresponding to a choice of multi-index $\gamma$. We first note that $x$ and $y$ are restricted to $\Omega$ by assumption. Secondly, the singularities in each variable $w_{j}$ are all locally integrable since near such singularities we have local behavior like $\frac{1}{|w_{j}|^{n-j-1}}$ or $\frac{1}{|\widehat{w_{j+1}} - w_{j}|^{n-1}}$. It remains to show bounds on the terms in the sum which involve integrations over one or more of the domains $\mathcal{D}_{j}^{3}$. Let $j$ be the first index where such an integration occurs in a given term. We may integrate out the variables $w_{1}\ldots w_{j-1}$ first since the domains of integration are bounded in those cases. Since $y_{j} \in \Omega$, we have $x - rw_{j} \prod_{l=j+1}^{m-1}|w_{l}| \in \Omega$. Thus
\begin{equation}
|w_{j}| \leq \frac{\mathrm{diam}(\Omega)}{r \prod_{l=j+1}^{m-1}|w_{l}|} =: M. \label{eq:wjbound}
\end{equation}

Suppose first that $n-j-1 \neq 1$. Then using (\ref{eq:wjbound}) we may bound each integral summand in (\ref{Amkernel}) by
\begin{align*}
& \int_{\mathcal{D}_{\gamma}'}\frac{C}{r^{n-m}|\eta - w_{m-1}|^{n-1}|\widehat{w_{m-1}} - w_{m-2}|^{n-1} \cdots |\widehat{w_{j+2}} - w_{j+1}|^{n-1} }\\
& \qquad \cdot \frac{1}{\prod_{l=j+1}^{m-1}|w_{l}|^{n- l -1}} \left| \left[\frac{1}{u^{n-j-2}} \right]_{\frac{1}{4}}^{M} \right|\,dw_{j+1}\,\cdots \,dw_{m-1}\\
& \leq \left| \frac{1}{r^{j+2-m}} - 4^{n-j-2}\right|\int_{\mathcal{D}_{\gamma}'}\frac{C}{|\eta - w_{m-1}|^{n-1}|\widehat{w_{m-1}} - w_{m-2}|^{n-1} \cdots |\widehat{w_{j+2}} - w_{j+1}|^{n-1} }\\
& \qquad \cdot \frac{ \mathrm{diam}(\Omega)^{2+j-n}}{\prod_{l=j+1}^{m-1}|w_{l}|^{j+1-l}} \,dw_{j+1}\,\cdots \,dw_{m-1},
\end{align*}
where $\mathcal{D}_{\gamma}'$ is the domain obtained from $\mathcal{D}_{\gamma}$ by removing the domains of integration with respect to $w_{1}\ldots w_{j}$. If on the other hand, $n-j-1 = 1$, then one obtains an estimate involving $\ln(r)$ and $\ln{|w_{l}|}$, which also yields a suitable bound on the integral, since $\ln(|w_{l}|)|w_{l}|^{s}$ is locally integrable at $0$ for any $s > -n$. Further integrations over domains $\mathcal{D}_{j'}^{3}$ will result in similar estimates, where the bound $M$ in (\ref{eq:wjbound}) will have the same kind of dependence on $r$. By Fubini's theorem $I_{m}$ is well defined as an iterated integral for $r \neq 0$.

It remains to show smoothness of $I_{m}$ with respect to $\eta$. Notice that when $\gamma_{m-1} = 1$ or $3$ (i.e. when $w_{m-1} \in B_{0}\left( \frac{1}{2} \right)$ or $w_{m-1} \in \rn{n} \setminus \left(B_{0}\left( \frac{1}{4} \right) \cup B_{\eta}\left( \frac{1}{4} \right)\right)$), then $|\eta - w_{m-1}|$ is uniformly positive. Thus, we can differentiate under the integral sign with respect to $\eta$ the term $\frac{1}{|\eta - w_{m-1}|^{n-1}}$ infinitely many times. On the other hand, to differentiate the terms with $\gamma_{m-1} = 2$ (i.e. $w_{m-1} \in B_{\eta}\left( \frac{1}{2} \right)$), we make the substitution $z_{m-1} = \eta - w_{m-1}$ and note that $z_{m-1} \in B_{0}\left( \frac{1}{2} \right)$. Then the term $\frac{1}{|w_{m-1}|^{n-m}} = \frac{1}{|\eta - z_{m-1}|^{n-m}}$ is uniformly bounded away from $0$, so we may differentiate it arbitrarily many times with respect to $\eta$. The only other potential problem will occur if in addition, $\gamma_{m-2} = 2$ or equivalently, $w_{m-2} \in B_{\widehat{w_{m-1}}}\left( \frac{1}{2} \right) = B_{\widehat{\eta - z_{m-1}}}\left( \frac{1}{2} \right)$. In this case, we make yet another change of variables $z_{m-2} = \widehat{w_{m-1}} - w_{m-2} = \widehat{\eta - z_{m-1}} - w_{m-2}$, so that $z_{m-2} \in B_{0}\left( \frac{1}{2} \right)$ and the term $\frac{1}{|\widehat{w_{m-1}} - w_{m-2}|^{n-1}}$ does not depend on $\eta$. Similarly, the term $\frac{1}{|w_{m-2}|^{n-m+1}} = \frac{1}{|\widehat{\eta - z_{m-1}} - z_{m-2}|^{n-m+1}}$ is uniformly bounded for $z_{m-2} \in B_{0}\left( \frac{1}{2} \right)$, and so we can differentiate it arbitrarily many times with respect to $\eta$. Continuing in this way, if at any point $\gamma_{j} = 2$ we make the change of variables $z_{j} = \widehat{w_{j+1}} - w_{j}$ and obtain similar bounds. The main idea is to transfer the derivative to the numerator function $\alpha_{m}$. It is a routine matter to check that under such change of variables, differentiation of $\alpha_{m}$ with respect to $\eta$ does not result in any unbounded factors involving $\eta$ via the chain rule.

\begin{remark}By the same reasoning, if $\alpha(x,y,r,\eta,\theta)$ is in $C_{c}^{2}(\rn{n} \times \rn{n} \times \mathbb{R}_{+} \times \mathbb{S}^{n-1} \times \mathbb{S}^{n-1})$, then $\beta_{m} \in C_{c}^{2}$.\end{remark}

We now consider the following adaptation of Lemma 2 from \cite{xraygeneric}:
\begin{lemma}
Let $m \geq 1$ and let $\mathcal{A}:C_{0}(\Omega) \to C(\Omega, \mathbb{S}^{n-1})$ be the operator
\begin{equation}
\mathcal{A}f(x,\theta) = \int_{\mathbb{S}^{n-1}}\int_{\mathbb{R}}r^{m-1}\alpha(x,r,\omega, \theta)f(x+r\omega)\,dr\,d\omega,
\end{equation}
where $\alpha \in C^{\infty}(\Omega \times \mathbb{R} \times \mathbb{S}^{n-1}\times \mathbb{S}^{n-1})$. Then $\mathcal{A}$ is a classical $\Psi$DO of order $-m$ with full symbol
\begin{equation*}
a(x,\xi) \sim \sum_{k=m-1}^{\infty}a_{k}(x,\xi),
\end{equation*}
where
\begin{equation*}
a_{k}(x,\xi) = 2\pi \frac{i^{k}}{k!} \int_{\mathbb{S}^{n-1}}\partial_{r}^{k}A(x,0,\omega,\theta)\delta^{(k)}(\omega \cdot \xi)\,d\omega.
\end{equation*}\label{pseudochar}
\end{lemma}
\begin{proof}
The proof directly follows from the proof of the case $m=1$ in \cite{xraygeneric}. Let $A'(x,r,\omega,\theta) = r^{m-1}A(x,r,\omega, \theta)$, and note that if $A'$ is an odd function of $(r,\omega)$, then $\mathcal{A}f = 0$. So we may replace $A' = r^{m-1}A$ by
\begin{align*}
A'_{\mathrm{even}}(r,\omega) & = \frac{1}{2}(A'(r,\omega) + A'(-r,-\omega)) \\
& = \frac{1}{2}( r^{m-1}A(r,\omega) + (-r)^{m-1}A(-r,-\omega)) \\
& = \frac{r^{m-1}}{2}(A(r,\omega) + (-1)^{m-1}A(-r,-\omega)).
\end{align*}
We can then integrate over $r \geq 0$ only and double the result. Thus,
\begin{equation*}
\mathcal{A}f(x) = 2\int_{\mathbb{S}^{n-1}}\int_{0}^{\infty}A'_{\mathrm{even}}(x,r,\omega,\theta)f(x+r\omega)\,dr\,d\omega.
\end{equation*}
Changing to polar coordinates via $z = r\omega$ and setting $y = x+r\omega$, we obtain
\begin{equation*}
\mathcal{A}f(x) = 2 \int A'_{\mathrm{even}}\left( x, |y-x|, \frac{y-x}{|y-x|}, \theta \right) \frac{f(y)}{|y-x|^{n-1}}\,dy.
\end{equation*}
We then use a finite Taylor expansion of $A_{\mathrm{even}}(x,r,\omega,\theta)$ in $r$ near $r =0$ with $N > 0$ to get
\begin{equation*}
A'_{\mathrm{even}}(x,r,\omega,\theta) = \sum_{k=0}^{N-1}A'_{\mathrm{even},k}(x,\omega, \theta)r^{k} + r^{N}R_{N}(x,r,\omega, \theta).
\end{equation*}
One can compute that $2A'_{\mathrm{even},k}(x,\omega,\theta) = A'_{k}(x,\omega, \theta) + (-1)^{k}A'_{k}(x,-\omega,\theta)$, where $k!A'_{k} = \partial_{r}^{k} \vert_{r=0}A' = \partial_{r}^{k} \vert_{r=0}\left[ r^{m-1}A\right]$. Clearly, $A'_{k} = 0$ for all $k <m-1$. Therefore, following the proof  of Lemma 2 in \cite{xraygeneric}, we obtain that the terms $a_{k}(x,\xi,\theta) = 2\pi i^{k} \int_{\mathbb{S}^{n-1}}A_{k}'(x,\omega,\theta)\delta^{(k)}(\omega \cdot \xi)\,d\omega = 0$ for all $k < m-1$.\qquad \end{proof}

\begin{proposition}Let $\alpha \in C_{c}^{\infty}(\rn{n} \times \rn{n} \times \mathbb{R} \times \mathbb{S}^{n-1} \times \mathbb{S}^{n-1})$ and consider the operator $A: L^{2}(\rn{n} \times \mathbb{S}^{n-1}) \to L^{2}(\rn{n}; C^{\infty}(\mathbb{S}^{n-1}))$ defined by
\begin{equation}
[Af](x,\theta) = \int \frac{\alpha(x,y,|x-y|, \widehat{x-y}, \theta)}{|x-y|^{n-1}}f(y, \widehat{x-y})\,dy.
\end{equation}
Then for $m \geq 1$, $A^{m}J$ is a classical pseudodifferential operator of order $-m$ with smooth parameter $\theta$, where $J:L^{2}(\rn{n}) \to L^{2}(\rn{n} \times \mathbb{S}^{n-1})$ is the extension operator $Jf(x,\theta) = f(x)$.\label{pseudoiterate}\end{proposition}

\begin{proof}Recall that
\begin{align*}
[A^{m}Jf](x,\theta) & = \int I_{m}(x,y,|x-y|,\widehat{x-y},\theta)f(y)\,dy\\
& = \int \frac{\beta_{m}(x,y,|x-y|,\widehat{x-y},\theta)}{|x-y|^{n-m}}f(y)\,dy\\
& = \int_{\mathbb{S}^{n-1}} \int_{0}^{\infty} r^{m-1}\beta_{m}(x,x-r\omega, r, \omega, \theta)f(x-r\omega)\,dr\,d\omega\\
& = \int_{\mathbb{S}^{n-1}} \int_{0}^{\infty} r^{m-1} \beta_{m}(x,x+r\omega, r, -\omega, \theta)f(x+r\omega)\,dr\,d\omega.
\end{align*}
We can then apply Lemma \ref{pseudochar} to complete proof.\qquad \end{proof}

Let us also compute the adjoint of the operator $A^{m}J$. Given functions $f \in L^{2}(\rn{n} \times \mathbb{S}^{n-1})$ and $g \in L^{2}(\rn{n})$, we have
\begin{align*}
& \langle [A^{m}J]^{*}f, g \rangle_{L^{2}(\rn{n})}\\
& = \langle f, A^{m}Jg \rangle_{L^{2}(\rn{n} \times \mathbb{S}^{n-1})}\\
& = \iiint\frac{\beta_{m}(x,y,|x-y|,\widehat{x-y},\theta)}{|x-y|^{n-m}}g(y)f(x,\theta)\,dy\,dx\,d\theta\\
& = \int g(y) \left( \int_{\rn{n}} \int_{\mathbb{S}^{n-1}} \frac{\beta_{m}(x,y,|x-y|,\widehat{x-y},\theta)}{|x-y|^{n-m}}f(x,\theta) \,dx\,d\theta \right)\,dy.
\end{align*}
Thus
\begin{equation}
[A^{m}J]^{*}f(x) = \int_{\rn{n}} \int_{\mathbb{S}^{n-1}}\frac{\beta_{m}(y,x,|y-x|,\widehat{y-x},\theta)}{|x-y|^{n-m}}f(y,\theta)\,d\theta\,dy.
\end{equation}
It is then clear that $[A^{m}J]^{*}: H^{l}(\Omega \times \mathbb{S}^{n-1}) \to H^{l+m}(\Omega)$ using a similar argument as in Lemma 2 of \cite{inversesource} together with Proposition \ref{pseudoiterate}.


\nocite{*}
\bibliographystyle{siam}
\bibliography{088507R}
\end{document}